\documentclass[11pt]{article}

\usepackage{setspace}

\usepackage{amsthm, amssymb, amstext}
\usepackage[fleqn]{amsmath}
\usepackage{latexsym}
\usepackage[dvips]{graphicx}
\usepackage{hyperref}

\usepackage[capitalise,nameinlink]{cleveref}
\usepackage{thm-restate}
\usepackage{tikz}
\usetikzlibrary{calc}
\usetikzlibrary{decorations.pathmorphing}
\usetikzlibrary{positioning}

\usepackage{mathtools}
\usepackage{enumerate} 
\usepackage{paralist}
\usepackage{enumitem}
\usepackage{dsfont}

\usepackage{todonotes}
\usepackage{comment}

\newcommand{\sm}{\setminus}

\definecolor{myRed}{rgb}{0.68, 0.05, 0.0}
\colorlet{myBlue}{blue!90!black}
\colorlet{myViolet}{myBlue!55!myRed}
\definecolor{darkraspberry}{rgb}{0.53, 0.15, 0.34}
\definecolor{olive}{rgb}{0.42, 0.56, 0.14}

\hypersetup{
	colorlinks=true,
    linkcolor=darkraspberry,
    citecolor=olive,
	bookmarksopen=true,
	bookmarksnumbered,
	bookmarksopenlevel=2,
	bookmarksdepth=3
}

\theoremstyle{definition}
\newtheorem{theorem}{Theorem}[section]
\newtheorem{lemma}[theorem]{Lemma}

\newtheorem{conjecture}[theorem]{Conjecture}

\theoremstyle{definition}

\usepackage{amsthm}
\usepackage{authblk}

\title{Unavoidable induced subgraphs in graphs with complete bipartite induced minors}

\author[1]{Maria Chudnovsky}
\author[2]{Meike Hatzel}
\author[3]{Tuukka Korhonen}
\author[4]{Nicolas~Trotignon}
\author[5]{Sebastian Wiederrecht}

\affil[1]{Princeton University, Princeton, NJ 08544, USA}
\affil[2]{Discrete Mathematics Group, Institute for Basic Science,
  Daejeon, South Korea}
\affil[3]{University of Copenhagen, Denmark}
\affil[4]{CNRS, ENS de Lyon, Université Claude Bernard Lyon 1, LIP,
  UMR~5668, Lyon, France}
\affil[5]{KAIST, Daejeon, South Korea}

\date{\today}

\begin{document}
\maketitle

\begin{abstract}
We prove that if a graph contains the complete bipartite graph $K_{134, 12}$ as an induced minor, then it contains a cycle of length at most~12 or a theta as an induced subgraph.  With a longer and more technical proof, we prove that if a graph contains $K_{3, 4}$ as an induced minor, then it contains a triangle or a theta as an induced subgraph. Here, a \emph{theta} is a graph made of three internally vertex-disjoint chordless paths $P_1 = a \dots b$, $P_2 = a \dots b$, $P_3 = a \dots b$, each of length at least two, such that no edges exist between the paths except the three edges incident
to $a$ and the three edges incident to $b$.

A consequence is that excluding a grid and a complete bipartite graph as induced minors is not enough to guarantee a bounded tree-independence number, or even that the treewidth is bounded by a function of the size of the maximum clique, because the existence of graphs with large treewidth that contain no triangles or thetas as induced subgraphs is already known (the so-called layered wheels). 
\end{abstract}

\section{Introduction}

Graphs in this paper are finite and simple. A graph $H$ is an
\emph{induced subgraph} of a graph $G$ if $H$ can be obtained from $G$
be repeatedly deleting vertices. It is an \emph{induced minor} of $G$
if $H$ can be obtained from $G$ be repeatedly deleting vertices and
contracting edges.  It is a \emph{minor} of $G$ if $H$ can be obtained
from $G$ be repeatedly deleting vertices, deleting edges and
contracting edges.  We denote by $K_t$ the complete graph on $t$ vertices and by $K_{a,b}$ the complete bipartite graph  with sides of size $a$ and $b$.  The $(k\times k)$-grid is the graph whose vertices are the pairs $(i, j)$ of integers such that $1\leq i, j \leq k$ and where $(i, j)$ is adjacent to $(i', j')$ if and only if $|i-i'| + |j-j'| = 1$.  

The tree-independence number was introduced
in~\cite{DBLP:journals/jctb/DallardMS24}. 
It is defined via tree-decompositions similarly to treewidth, except that the number associated to each bag of a tree-decomposition is the maximum size of an independent set of the graph induced by the bag, instead of its number of vertices (we omit the full definition for the sake of brevity).
It attracted some attention lately, in particular because for each class of graphs with bounded
tree-independence number, there exists polynomial time algorithms for maximum independent set and other related problems~\cite{DBLP:journals/jctb/DallardMS24,DBLP:conf/icalp/DallardFGKM24,DBLP:conf/esa/LimaMMORS24,DBLP:conf/soda/Yolov18}.

The celebrated grid minor
theorem of Robertson and Seymour~\cite{DBLP:journals/jct/RobertsonS86}
states that there exists a function
$f \colon \mathds N \rightarrow \mathds N$ so that any graph with treewidth
at least $f(k)$ contains a $(k \times k)$-grid as a minor.  The
motivation for our work is the quest for a similar theorem with
``tree-independence number'' instead of ``treewidth''.

Here, the natural containment relation should be ``induced minor''
instead of ``minor'', since the tree-independence number is monotone
under taking induced minors but not under taking  minors.  The
following lemma shows that for triangle-free graphs, treewidth and
tree-independence number are essentially equivalent. 

\begin{lemma}[see \cite{DBLP:journals/jctb/DallardMS24}]
  \label{l:ramsey}
  The treewidth of a triangle-free graph with tree-independence number
  at most $t$ is at most $R(3, t+1)-2$ where $R(a, b)$ denotes the
  classical Ramsey number.
\end{lemma}

By \cref{l:ramsey}, grids and complete bipartite graphs have unbounded
tree-independence number.  The list of unavoidable graphs arising from
a large tree-independence number should therefore contain at least
large grids and large complete bipartite graphs.  Our main result is
that this list is not complete. Indeed, we prove that a construction
called \emph{layered wheel}, first defined
in~\cite{DBLP:journals/jgt/SintiariT21}, contains no
$(5\times 5)$-grid and no $K_{3, 4}$ as an induced minor, while having
arbitrarily large tree-independence number.

\subsection*{Outline of the proof}

To avoid too heavy notation, we allow some abuse. Typically, we do not
distinguish between a set of vertices in graph $G$ and the subgraph of
$G$ that it induces.  For instance, when $P$ is a path and $v$ is a
vertex in a graph $G$, we denote by $P\sm v$ either $V(P) \sm \{v\}$
or $G[V(P) \sm \{v\}]$.  Also, we may say that a set of vertices $C$
of $G$ is \emph{connected} when the correct statement should be that
$G[C]$ is connected.  We hope that this improves the readability
without causing any confusion.

We do not need to define layered wheels, which is good since the
definition is a bit long.  We just need some of their properties and some preliminary definitions to state them.  A \emph{theta} is a graph made
of three internally vertex-disjoint chordless paths $P_1 = a \dots b$,
$P_2 = a \dots b$, $P_3 = a \dots b$, each of length at least two, such
that no edges exist between the paths except the three edges incident
to $a$ and the three edges incident to $b$.  A graph is
\emph{theta-free} if it does not contain a theta as an induced subgraph,
and more generally, a graph is \emph{$H$-free} whenever it does not
contain $H$ (when $H$ is a graph) or any graph in $H$ (when $H$ is a class of graphs, such as thetas) as an induced subgraph.  The only property of layered wheels that we need is the following theorem that states their existence.

\begin{theorem}[see \cite{DBLP:journals/jgt/SintiariT21}]
  \label{th:lw}
  For all integers $t\geq 1$ and $k\geq 3$, there exists a theta-free graph of girth $k$ that contains $K_t$ as a minor.
\end{theorem}

Since containing $K_t$ as a minor implies having treewidth at
least~$t$, layered wheels provide theta-free graphs of arbitrarily
large girth and treewidth. Their tree-independence number is also
arbitrarily large by \cref{l:ramsey}.  To fulfill our goal, it
therefore remains to prove that layered wheels do not contain large
grids or complete bipartite graphs as induced minors. As far as we can
see, this is non-trivial because even if layered wheels are precisely
defined, checking directly that they do not contain some
$(k\times k)$-grid or $K_{r, s}$ as an induced minor seems to be
tedious, at least according to our several attempts. An indication of
this is that some layered wheels do contain $K_{3, 3}$ as an induced
minor, which is not obvious, see \cref{f:k33Layered} (this figure is
meaningful only with the precise definition of a layered wheel).

\begin{figure}[!ht]
    \begin{center} 
    \includegraphics{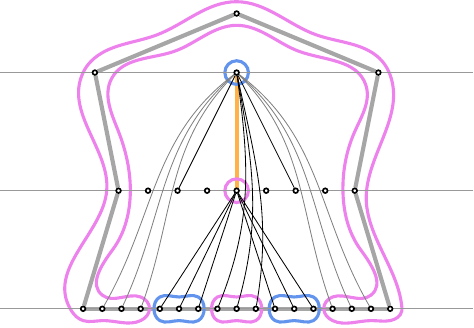}
    \end{center}
    \caption{A $K_{3, 3}$ induced minor in a layered wheel}
    \label{f:k33Layered}
\end{figure}

Our approach is therefore less direct: we study what induced subgraphs
are forced by the presence of a large grid or complete bipartite graph as an induced minor as we explain now.

\subsection*{Complete bipartite graphs}

The list of induced subgraphs forced by the presence of $K_{2, 3}$ as an induced minor is already known and not very difficult to obtain, see~\cite{DBLP:conf/iwoca/DallardDHMPT24}. But this is not enough for our purpose since containing $K_{2, 3}$ as an induced minor does not imply anything we can use. By a short argument, we first prove the following.

\begin{restatable}{theorem}{short}
  \label{th:134}
  If a graph $G$ contains $K_{134, 12}$ as an induced minor, then $G$
  contains a cycle of length at most~12 or a theta as an induced
  subgraph. 
\end{restatable}

The advantage of \cref{th:134} is its short proof. But its statement
is far from being optimal, as shown by the following, that we obtain
by a more careful structural study.

\begin{theorem}
  \label{th:k34ThetaTriangle}
  If a graph $G$ contains $K_{3, 4}$ as an induced minor, then $G$
  contains a triangle or a theta as an induced subgraph. 
\end{theorem}

Once \cref{th:134} (or~\ref{th:k34ThetaTriangle}) is proved, checking
that layered wheels contain no $K_{134, 12}$ (or no $K_{3, 4}$) as an
induced minor becomes trivial since by \cref{th:lw}, they
do not contain short cycles and thetas as induced subgraphs.  Our results therefore avoid some tedious checking, but we also believe
that they are of independent interest.

Note that \cref{th:k34ThetaTriangle} is best possible in several ways.
First, thetas need to be excluded because the subdivisions of
$K_{3, 4}$ provide triangle-free graphs that obviously contain
$K_{3, 4}$ as an induced minor while containing thetas and no
triangles.  Triangles must be excluded because of line graphs of
subdivisions of $K_{3, 4}$, see \cref{f:line} where the line graph of
the graph obtained by subdividing each edge of $K_{3, 4}$ is
represented: it is a theta-free graph that obviously contains
$K_{3, 4}$ as an induced minor.  A less obvious construction is
represented in \cref{f:best}, showing that $K_{3, 4}$ cannot be
replaced by $K_{3, 3}$ in \cref{th:k34ThetaTriangle}.  Also, for all
$\ell \geq 1$, there exist (theta, triangle)-free graphs that contain
$K_{2, \ell}$ as an induced minor, see \cref{f:K2l}

\begin{figure}[!ht]
    \begin{center}
       \includegraphics[width=9cm]{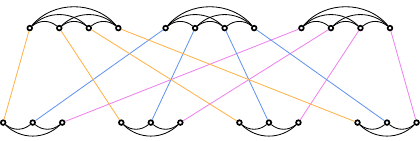}
    \end{center}
    \caption{The line graph of the graph obtained by subdividing each
      edge of $K_{3, 4}$}
    \label{f:line}
\end{figure}

\begin{figure}[!ht]
    \begin{center}
        \includegraphics[width=5cm]{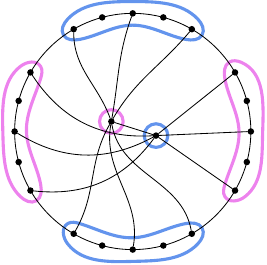}
    \end{center}
    \caption{A (triangle, theta)-free graph containing $K_{3, 3}$ as an
    induced minor}
    \label{f:best}
\end{figure}

\begin{figure}[!ht]
    \begin{center}
        \includegraphics[width=8cm]{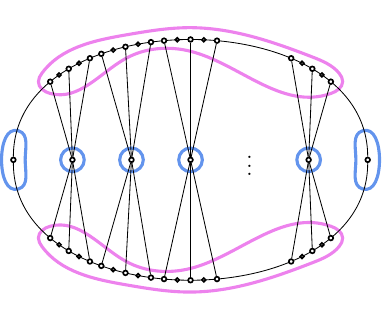}
    \end{center}
    \caption{A (theta, triangle)-free graph containing $K_{2, \ell}$ as
      an induced minor}
    \label{f:K2l}
\end{figure}

In fact, we prove a more general result than \cref{th:k34ThetaTriangle}.
Before stating it, we need some definitions.

A \emph{prism} is a graph made of three vertex-disjoint chordless
paths $P_1 = a_1 \dots b_1$, $P_2 = a_2 \dots b_2$,
$P_3 = a_3 \dots b_3$, such that $\{a_1,a_2,a_3\}$ and $\{b_1,b_2,b_3\}$ are triangles and no edges
exist between the paths except those of the two triangles and either:

\begin{itemize}
 \item $P_1$, $P_2$ and $P_3$ all have length at least~1, or 
 \item one of $P_1, P_2$, $P_3$ has length~0 and each of the other two has length at least~2.
\end{itemize}
 
Note that allowing a path of length zero in prisms (for instance if $a_1=b_1$)
is not standard, but natural in our context.
A prism with a path of length zero is sometimes referred to as a \emph{line wheel},
but we do not use this terminology here.

\begin{figure}[!ht]
    \centering
    \includegraphics{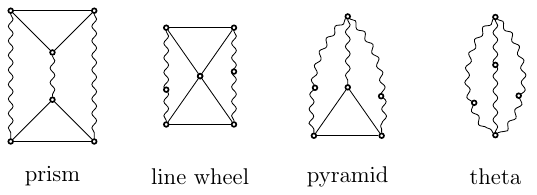}
    \caption{The different $3$-path configurations}
    \label{fig:3PC}
\end{figure}

A \emph{pyramid} is a graph made of three chordless paths
$P_1 = a \dots b_1$, $P_2 = a \dots b_2$, $P_3 = a \dots b_3$, each of
length at least one, and two of them with length at least two,
vertex-disjoint except at $a$, and such that $\{b_1,b_2,b_3\}$ is a
triangle and no edges exist between the paths except those of the
triangle and the three edges incident to~$a$.

A \emph{hole}
in a graph is a chordless cycle of length at least~4.
A graph that is either a theta, a prism or a pyramid is called a
\emph{3-paths configuration}, or \emph{3PC} for short.   Observe that
a graph $G$ is a 3PC if and only if there exist three pairwise
internally vertex disjoint paths $P_1$, $P_2$, $P_3$ such that $V(G) =
V(P_1) \cup V(P_2) \cup V(P_3)$ and for every distinct $i, j\in \{1, 2, 3\}$, $V(P_i) \cup V(P_j)$ induces a hole. 

\begin{restatable}{theorem}{longer}
  \label{th:k343PC}
  If a graph $G$ contains $K_{3, 4}$ as an induced minor, then $G$
  contains a 3PC as an induced subgraph. 
\end{restatable}

\cref{th:k343PC} clearly implies \cref{th:k34ThetaTriangle} because a
3PC that is not a theta contains a triangle.  Proving \cref{th:k343PC}
is just slightly longer than \cref{th:k34ThetaTriangle}, but is also
interesting in its own right for the following reason.   A hole is \emph{even} if
it has an even number of vertices. Computing the
maximum independent set of an even-hole-free graph in polynomial time
is a well known open question, see~\cite{vuskovic:evensurvey}.  Hence, understanding the
tree-independence number of even-hole-free graphs would be
interesting. Moreover, \cref{th:k343PC} implies directly the existence of
even-hole-free graphs of large tree-independence number that contain
no large grids and no $K_{3, 4}$ as induced minors.
Indeed,~\cite{DBLP:journals/jgt/SintiariT21} not only provides the
layered wheels that we already mentioned, but a variant, called
\emph{even-hole-free layered wheels}, whose existence is stated in the next theorem.

\begin{theorem}[see \cite{DBLP:journals/jgt/SintiariT21}]
  \label{th:ehf-lw}
  For all integers $t\geq 1$, there exists a ($K_4$, even hole,
  3PC)-free graph that contains $K_t$ as a minor.
\end{theorem}

So, the simple counter-part of the Robertson and Seymour grid theorem,
that would state that graphs with large tree-independence number
should contain a large grid or a large complete bipartite graph as an
induced minor is false, even when we restrict ourselves to
even-hole-free graphs.  Even-hole-free layered wheels provide a
counter-example (note that by \cref{l:ramsey}, being $K_4$-free
ensures that the high treewidth implies a high tree-independence
number).

\cref{th:k343PC} is best possible in some sense as we explain
now. Thetas are essential in the statement because subdivisions of
$K_{3, 4}$ are easily seen to contain thetas.  Also prisms are needed
because of line graphs of
subdivisions of $K_{3, 4}$ that contain prisms (see \cref{f:line}).  We have to
allow a path of length zero in a prism because of the graph depicted
in \cref{f:3PC-prismZero} that contains $K_{3, 4}$ as an induced minor
while the only 3PC in it is a prism with a path of length zero.  Maybe
pyramids are not needed in the statement, but we need them in the
proof for reasons that we now explain.

\begin{figure}[!ht]
  \begin{center}
 \includegraphics[width=7cm]{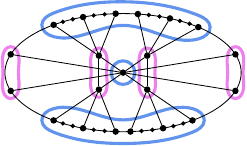}
 \caption{$K_{3,4}$ as an induced minor while all 3PC's are prisms
   with a path of length zero\label{f:3PC-prismZero}}
 \end{center}
\end{figure}

The proof of \cref{th:k343PC} relies on a precise description of how
$K_{3, 3}$ can be contained as an induced minor in a 3PC-free graph,
see \cref{l:exact1}.  The description is too long to be stated in the
introduction, but \cref{f:situations} provides representations of the
different possible situations.

\begin{figure}[!ht]
  \begin{center}
    \includegraphics[width=0.8\textwidth]{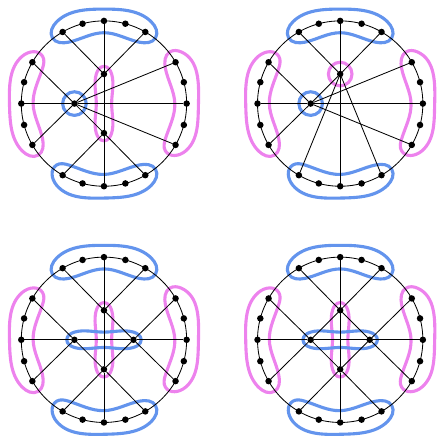}
        \caption{$K_{3,3}$ as an induced minor in 3PC-free graphs}
    \label{f:situations}
 \end{center}
\end{figure}

 Observe that excluding pyramids is necessary in \cref{l:exact1},
 because of the graph presented in \cref{f:3PC-pyramid}, that
 contains $K_{3, 3}$ as induced minor, while the only 3PC in it is a
 pyramid.

\begin{figure}[!ht]
  \begin{center}
 \includegraphics{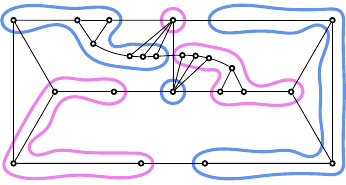}
 \caption{$K_{3, 3}$ as an induced minor in a (theta, prism)-free
   graph\label{f:3PC-pyramid}}
 \end{center}
\end{figure}

\subsection*{Grids}

Grids are easier to handle than complete bipartite graphs as shown by the next lemma whose proof is less involved than the proofs of Theorems~\ref{th:134}, \ref{th:k34ThetaTriangle} or~\ref{th:k343PC}. We do not know if a $(4\times 4)$-grid as an induced minor is enough to guarantee the presence of a 3PC as an induced subgraph.  

\begin{restatable}{lemma}{lfivefive}
\label{l:55}
    If a graph $G$ contains a $(5\times 5)$-grid as an induced minor, then $G$ contains a 3PC as an induced subgraph.
\end{restatable}

 Checking
that layered wheels (and their even-hole-free variant) contain no $(5\times 5)$-grid as an induced minor is then easy by \cref{th:lw} and \cref{th:ehf-lw}. 

\subsection*{Outline of the paper}

In \cref{sec:type}, we prove some technical lemmas about how a
connected induced subgraph of some graph $G$ sees the rest of $G$.
These lemmas are more or less  known already. We reprove them for the sake of completeness and because we could not find them with the precise
statement that we need.  Note that they are needed for all the other results.  In \cref{sec:55}, we prove \cref{l:55}. In \cref{sec:134}, we
prove \cref{th:134}.  In \cref{sec:33}, we prove \cref{th:k343PC}.  We conclude the paper by \cref{sec:open} that presents several open questions.

\subsection*{Notation}

Let $G$ be a graph and $X$ and $Y$ be disjoint sets of vertices of $G$.
We say that $X$ is \emph{complete} to $Y$ if every vertex of $X$
is adjacent to every vertex of $Y$, $X$ is \emph{anticomplete} to $Y$
if  no vertex of $X$ is adjacent to a vertex of $Y$.  We say that
$X$ \emph{sees} $Y$ if there exist $x\in X$ and $y \in Y$ such that
$xy\in E(G)$.  Note that the empty set is complete (and anticomplete)
to every set of vertices of $G$.  

By \emph{path} we mean a sequence of vertices $p_1 \dots p_k$ such
that for all $1\leq i<j \leq k$, $p_ip_j \in E(G)$ if and only if $j=i+1$.
Therefore, what we call \emph{path} for the sake of brevity
is sometimes referred to as \emph{chordless path} or \emph{induced path} in a more standard notation.
We use the notation $aPb$ to denote
the subpath of $P$ from $a$ to $b$ (possibly $a=b$ since a path may
consist of a single vertex).

When we deal with a theta with two vertices of degree~3 $u$ and $v$,
we say that the theta is \emph{from $u$ to $v$}.
We use similar terminology for prisms and pyramids that are, respectively, from a triangle to a triangle and from a vertex to a triangle.

\section{Types}
\label{sec:type}

\begin{figure}[!ht]
    \centering
    \resizebox{\textwidth}{!}{\includegraphics{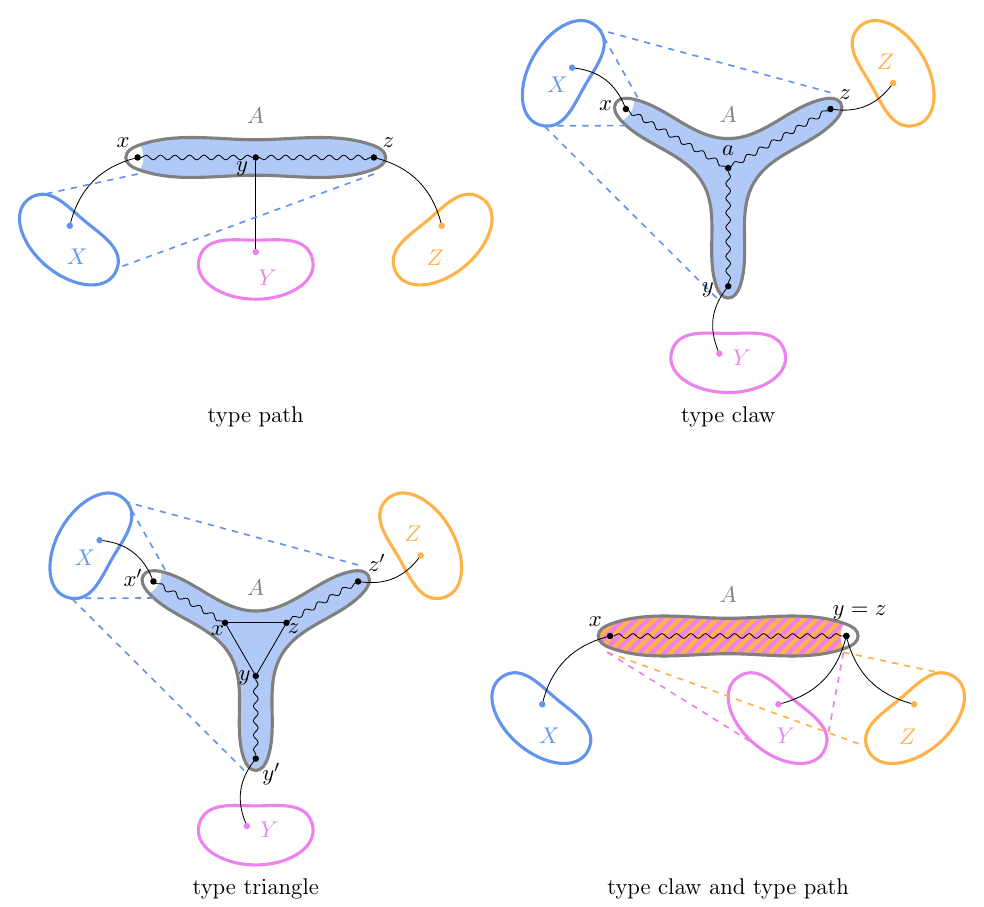}}
    \caption{We illustrate the three types and a degenerate case in which type path and type claw overlap.
    For the three types we by way of example illustrate how $X$ is not connected to anything in $A$ but $x$.
    For the degenerate it is depicted that $Y$ and $Z$ are not connected to anything in $A$ but the vertex $y=z$.}
    \label{fig:types}
\end{figure}

Let $A$, $X$, $Y$ and $Z$ be four disjoint sets of vertices in a
graph $G$.
We distinguish three different types the set $A$ can have, see \cref{fig:types} for an illustration.

We say that $A$ is of \emph{type path centered at $Y$ with respect to
$X$, $Y$ and $Z$} if there exists a path $P= x\dots z$ in $A$ such
that $x$ sees $X$, $z$ sees $Z$, $P$ sees $Y$, $P\sm x$ is
anticomplete to $X$ and $P\sm z$ is anticomplete to $Z$.

Similarly we define $A$ being of \emph{type path centered at $X$} and
being of \emph{type path centered at $Z$}.
We say that $A$ is of \emph{type path} with respect to $X$, $Y$ and $Z$ if it is of type
path centered at either $X$, $Y$ or $Z$.
Additionally, if $A$ is of type path centered at a set $W$, we call $W$ a \emph{center for $A$}.

Observe that when $A$ is of type path centered at $Y$ and a unique
vertex of $P$ sees $Y$, and moreover this vertex is $x$, then $A$ is
also centered at $X$. In particular, if $x=z$, then $A$ is centered at
$X$, $Y$ and $Z$.

We say that $A$ is of \emph{type claw} with respect to $X$, $Y$ and
$Z$ if there exists in $A$ a vertex $a$ and three paths
$P = a\dots x$, $Q= a\dots y$ and $R= a\dots z$ such that
$V(P) \cap V(Q) \cap V(R) = \{a\}$, $P\sm a$, $Q \sm a$ and $R \sm a$
are pairwise anticomplete, $x$ sees $X$, $y$ sees $Y$, $z$ sees $Z$,
$(P \cup Q \cup R) \sm x$ is anticomplete to~$X$,
$(P \cup Q \cup R) \sm y$ is anticomplete to $Y$ and
$(P \cup Q \cup R) \sm z$ is anticomplete to $Z$.

Observe that possibly $a=x$, $a=y$ or $a=z$ (in fact, possibly two or
three of these equalities hold).
Observe that if $a=x$, then $A$ is not only of type claw, but also of type path centered at $X$, this case is additionally depicted in \cref{fig:types}.

We say that $A$ is of \emph{type triangle} with respect to $X$, $Y$
and $Z$ if there exists in $A$ three vertex-disjoint paths
$P = x\dots x'$, $Q= y\dots y'$ and $R= z\dots z'$ such that the only
edges between $P$, $Q$ and $R$ are $xy$, $yz$ and $yz$, $x'$ sees $X$,
$y'$ sees $Y$, $z'$ sees $Z$, $(P \cup Q \cup R) \sm x$ is
anticomplete to~$X$, $(P \cup Q \cup R) \sm y$ is anticomplete to $Y$
and $(P \cup Q \cup R) \sm z$ is anticomplete to $Z$.

Observe that each of $P$, $Q$ and $R$ is possibly of length~0.

\begin{lemma}
  \label{l:type}
  Let $G$ be a graph and $A$, $X$, $Y$ and $Z$ be disjoint sets of
  vertices of $G$.  If $A$ is connected and $A$ sees $X$, $Y$ and $Z$,
  then $A$ is of type path, of type claw, or of type triangle with
  respect to $X$, $Y$ and $Z$.
\end{lemma}

\begin{proof}
  Suppose that $A$ is not of type path.  Let $P =x \dots z$ be a path
  in $A$ such that $x$ sees $X$, $z$ sees $Z$, $P\sm x$ is
  anticomplete to $X$ and $P\sm z$ is anticomplete to $Z$. Note that
  such a path exists (consider for instance a shortest path from the
  vertices that see $X$ to the vertices that see $Z$).  Since $A$ is not
  of type path centered at $Y$, $P$ is anticomplete to $Y$.  Let
  $Q = y\dots y'$ be path in $A$ disjoint from $P$ and such that $y$
  sees $Y$, $y'$ sees $P$, $Q \sm y$ is anticomplete to $Y$ and
  $Q\sm y'$ is anticomplete to $P$.  Note again that such a path
  exists (consider for instance a shortest path from the vertices that
  see $Y$ to the vertices that see $P$).  We suppose that $P$ and $Q$
  are chosen subject to the minimality of $V(P) \cup V(Q)$.

  Let $a$ be the neighbor of $y'$ in $P$ closest to $x$ along $P$ and
  $a'$ be the neighbor of $y'$ in $P$ closest to $z$ along $P$
  (note that $a=a'$ or  $aa'\in E(G)$ is possible).

  If $Q$ sees both $X$ and $Z$, then consider vertex $u$ in $Q$ such that $yQu$ sees both $X$ and $Z$, and choose $u$ closest to $y$ along $Q$. The path $yQu$ shows that $A$ is of type path (centered at $X$ or $Z$, possibly both, possibly also centered at $Y$ if $u=y$), a contradiction. Hence, we may assume up to symmetry that $Q$ is anticomplete to $Z$. 
  If $Q$ sees $X$, then the path $yQy'a'Pz$ shows that $A$ is of type
  path centered at $X$, a contradiction.  Hence, $Q$ is anticomplete
  to~$X$ and $Z$.

  If $a\neq a'$ and $aa'\notin E(G)$, then  consider the path $P' = x P a y' a' P z$.  If  $y=y'$, then because of $P'$, $A$ is of type path centered at $Y$, a  contradiction.  So, $y\neq y'$ and the paths $P'$ and $Q\sm y'$ contradict the minimality of $V(P) \cup V(Q)$.

  Hence $a=a'$ or $aa'\in E(G)$.  If $a=a'$, then the three paths
  $aPx$, $ay'Qy$ and $aPz$ show that $A$ is of type claw with respect
  to $X$, $Y$ and $Z$.  If $aa'\in E(G)$, then the three paths $aPx$,
  $Q$ and $a'Pz$ show that $A$ is of type triangle with respect to $X$,
  $Y$ and $Z$.
\end{proof}

The next lemma is implicitly about the presence of $K_{2, 3}$ as an
induced minor in a graph. In~\cite{DBLP:conf/iwoca/DallardDHMPT24},
it is proved along similar lines that a graph contains $K_{2, 3}$ as
an induced minor if and only if it contains some configuration from a
restricted list as an induced subgraph (the so-called thetas,
pyramids, long prisms and broken wheels,
see~\cite{DBLP:conf/iwoca/DallardDHMPT24} for the definitions).

\begin{lemma}
  \label{l:OnePath}
  Let $G$ be a 3PC-free graph and $A$, $B$, $X$, $Y$ and $Z$ be
  disjoint connected subsets of $V(G)$. If $X$, $Y$ and $Z$ are
  pairwise anticomplete, $A$ and $B$ are anticomplete to each other,
  and each of $A$ and $B$ sees each of $X$, $Y$ and $Z$, then $A$ and $B$ are
  of type path with respect to $X$, $Y$ and $Z$.
\end{lemma}

\begin{proof}
  Suppose for a contradiction that $A$ (the argument for $B$ is the same by symmetry) is not of type path with
  respect to $X$, $Y$ and $Z$.  Hence, by Lemma~\ref{l:type}, $A$ we
  may consider the two cases below.

  \noindent{\textbf{Case 1:}} $A$ is of type claw (and not of type path).  So, there
  exists a vertex $a\in A$ and three paths $P = a\dots x$,
  $Q = a\dots y$ and $R= a\dots z$ like in the definition of type
  claw.  Moreover, $a\neq x$, $a\neq y$ and $a\neq z$ for otherwise,
  $A$ would be of type path.

  If $B$ is of type claw, then there exist a vertex $b\in B$ and
  three paths $P' = b\dots x'$, $Q' = b\dots y'$ and $R'= b\dots z'$
  like in the definition of type claw. Consider a shortest path $P''$
  from $x$ to $x'$ with interior in $X$, and let $Q'' = y\dots y'$ and
  $R'' = z \dots z'$ be defined similarly through $Y$ and $Z$.  The
  nine paths $P$, $Q$, $R$, $P'$, $Q'$, $R'$, $P''$, $Q''$ and $R''$
  form a theta from $a$ to $b$, a contradiction.

  If $B$ is of type triangle, a similar contradiction is found because
  of the existence of a pyramid in $G$. 

  So $B$ is not of type claw or triangle.  By Lemma~\ref{l:type}, $B$
  is of type path, and up to symmetry we suppose that it is centered
  at $Y$.  Let $P' = x' \dots z'$ be a path like in the definition of
  type path. Consider a shortest path $P''$ from $x$ to $x'$ with
  interior in $X$, and let $R'' = z \dots z'$ be defined similarly
  through $Z$.  Let $Q''=b\dots b'$ be a shortest path in $Y$ such
  that $by\in E(G)$ and $b'$ sees~$P'$.  Let $a'$ be the neighbor of
  $b'$ in $P'$ closest to $x'$ along $P'$. Let $a''$ be the neighbor
  of $b'$ in $P'$ closest to $z'$ along $P'$.  If $a'=a''$, then $B$
  is of type claw, a contradiction.  If $a'a''\in E(G)$, then the
  seven paths $P$, $Q$, $R$, $P'$, $P''$, $Q''$ and $R''$ form a
  pyramid from $a$ to $b'a'a''$, a contradiction.  Hence $a\neq a'$
  and $aa'\notin E(G)$.  So the paths $P$, $Q$, $R$, $x'P'a'$,
  $a''P'z'$, $P''$, $Q''$ and $R''$ form a theta from $a$ to $b'$, a
  contradiction.

  \noindent{\textbf{Case 2:}} $A$ is of type triangle.

  The proof is almost the same as in Case~1, so we just sketch it.  If
  $B$ is of type claw, then $G$ contains a pyramid. If $B$ is of type
  triangle, then $G$ contains a prism. And if $B$ is of type path,
  then $G$ contains a prism or a pyramid.   
\end{proof}

Let $G$ and $H$ be two graphs. An \emph{induced minor model} of $H$ in
$G$ is a collection of pairwise disjoint sets $\{ X_v\}_{v\in V(H)},$
called the \emph{branch sets} of the model, such that

\begin{itemize}
\item $X_v\subseteq V(G)$ for all $v\in V(H),$
\item $X_v$ induces a connected subgraph of $G$ for every $v\in V(H),$
  and
\item $X_u$ sees $X_v$ (in $G$) if and only if $uv\in E(H)$.
\end{itemize}

It is well known and easy to check that $G$ contains a graph
isomorphic to $H$ as an induced minor if and only if there exists an
induced minor model of $H$ in $G$.  We identify an induced minor model
$\{ X_v\}_{v\in V(H)}$ of $H$ in $G$ with the graph
$H'\coloneqq G[\bigcup_{v\in V(H)}X_v]$.  Please note that the
definition of the branch sets is not uniquely determined by $H'.$

An induced minor model $H'\subseteq G$ of $H$ is \emph{minimal} if
$H' \sm a$ does not contain an induced minor isomorphic to $H$ for all
$a\in V(H').$

We denote by $K^{*}_{2,3}$ the graph obtained from $K_{2,3}$ by
subdividing every edge once, i.e., $K^{*}_{2,3}$ is the graph with two
degree-3 vertices and three paths of length four between them.

\begin{lemma}
\label{lem:subk23im}
If a  graph contains $K^{*}_{2,3}$ as an induced minor,
then it contains a 3PC.
\end{lemma}

\begin{proof}
  Suppose for a contradiction that a 3PC-free graph $G$ contains
  $H = K^{*}_{2,3}$ as an induced minor.  Denote by $a$ and $b$ the
  degree-3 vertices and let $ap_1p_2p_3b$, $aq_1q_2q_3b$ and
  $ar_1r_2r_3b$ be the three paths of $H$.  Consider a minimal induced
  minor model $\{X_v\}_{v \in V(H)}$ of $H$ in $G$. For all
  $v\in V(H) \sm \{a, b\}$, $v$ has two neighbors $u$ and $w$ in $H$
  and by minimality, $X_v$ is a path $P= v'\dots v''$ such that $v'$
  sees $X_u$, $v''$ sees $X_w$, $P\sm v'$ is anticomplete to $X_u$ and
  $P\sm v''$ is anticomplete to $X_w$.  It follows that if we set
  $A= X_a \cup X_{p_1} \cup X_{p_2} \cup X_{q_1} \cup X_{q_2} \cup
  X_{r_1} \cup X_{r_2}$, $X = X_{p_3}$, $Y = X_{q_3}$, $Z = X_{r_3}$
  and $B= X_b$, $A$ cannot be of type path with respect to $X$, $Y$
  and $Z$.  Indeed, since $X_{p_1}\cup X_{p_2}$, $X_{q_1}\cup X_{q_2}$
  and $X_{r_1}\cup X_{r_2}$ all induce paths of length at least~1 in
  $G$, $A$ cannot contain a path with vertices seeing $X$, $Y$ and
  $Z$.  Hence, $A$, $B$, $X$, $Y$ and $Z$ contradict \cref{l:OnePath}.
\end{proof}

\section{Proof of \cref{l:55}}
\label{sec:55}

We here prove \cref{l:55} that we restate. 

\lfivefive*

\begin{proof}
    If a graph contains a $(5\times 5)$-grid as an induced minor, then it contains $K_{2, 3}^*$ as an induced minor, because the $(5 \times 5)$-grid contains $K_{2, 3}^*$ as an induced subgraph.  The result therefore follows from \cref{lem:subk23im}. 
\end{proof}

\section{Proof of \cref{th:134}}
\label{sec:134}

\begin{lemma}
\label{lem:girthtree}
Let $G$ and $H$ be graphs and $\{X_v\}_{v \in V(H)}$ a minimal induced
minor model of $H$ in $G$.  For every $v \in V(H)$, the graph $G[X_v]$
does not contain cycles longer than the degree of $v$.
\end{lemma}
\begin{proof}
  First, for every neighbor $u$ of $v$ in $H$, let us mark a single
  vertex in $X_v$ that is adjacent to a vertex in $X_u$.  If
  there exists a connected induced proper subgraph of $G[X_v]$ that
  contains all of the marked vertices, we contradict the minimality of
  the induced minor model.

Suppose that $G[X_v]$ contains a cycle of length longer than the
degree of $v$, and let $C \subseteq X_v$ be the vertices of this
cycle.  We say that a vertex $w \in C$ is \emph{necessary} if either
$w$ is marked or there exists a connected component $Y$ of
$G[X_v \setminus C]$ that contains a marked vertex and whose only
neighbor in $C$ is the vertex $w$.  Because $|C|$ is larger than the
number of marked vertices, there exists a vertex in $C$ that is not
necessary, let $w \in C$ be such a vertex.  We remove from $X_v$ the
vertex $w$ and every component of $G[X_v \setminus C]$ whose only
neighbor in $C$ is $w$.  All of the remaining vertices in $X_v$ are
still connected to $C \setminus \{w\}$ and contain all of the marked
vertices, so we get a connected induced proper subgraph of $G[X_v]$
that contains all of the marked vertices, which is a contradiction.
\end{proof}

\Cref{lem:girthtree} is useful for asserting that $G[X_v]$ is a
tree, which in turn is useful for obtaining degree-1 vertices in $G[X_v]$.
We make use of degree-1 vertices with the following lemma.

\begin{lemma}
\label{lem:deg1inbranch}
Let $\{X_v\}_{v \in V(H)}$ be a minimal induced minor model of a graph
$H$ in a graph $G$, and let $v \in V(H)$.  For every vertex
$w \in X_v$ whose degree is one in $G[X_v]$, there exists
$u \in V(H) \setminus \{v\}$ so that $w$ is the only neighbor of $X_u$
in $X_v$.
\end{lemma}
\begin{proof}
  Otherwise, we could remove $w$ from $X_v$ and contradict the
  minimality of $\{X_v\}_{v \in V(H)}$.
\end{proof}

We say that such a branch set $X_u$ is \emph{private} to the vertex $w$.
We may now prove \cref{th:134} which we restate below. 

\short*

\begin{proof}
  Let $p = 12$ and $t = 2 \binom{p}{2}+2 = 134$.  Let $G$ be a graph
  with girth at least $p+1$, and let $\{X_v\}_{v \in V(K_{t,p})}$ be a
  minimal induced minor model of $K_{t,p}$ in $G$.  We denote the
  branch sets corresponding to vertices of $K_{t,p}$ on the side with
  $t$ vertices by $A_1, \ldots, A_t$ and on the other side by
  $B_1, \ldots, B_p$.  Note that $G$ is triangle-free.

  First suppose that there are two branch sets $A_i$, $A_j$ of size
  $|A_i|, |A_j| \le 2$.  In this case, there must be a vertex
  $v \in A_i$ that is adjacent to at least $p/2$ different branch sets
  on the other side, and a vertex $u \in A_j$ that is adjacent to at
  least $p/4 = 3$ different branch sets on the other side that $v$ is
  also adjacent to.  Fix three different branch sets $B_a$, $B_b$,
  $B_c$ that both $v$ and $u$ are adjacent to.  Now, by selecting in
  each branch set $B_a$, $B_b$, $B_c$ a shortest path from a neighbor
  of $v$ to a neighbor of $u$, we obtain a theta from $u$ to $v$.

  We may therefore assume that at least $t-1$ of the branch sets
  $A_1, \ldots, A_t$ contain at least three vertices.
  By \cref{lem:girthtree} and because $G$ has girth at least $p+1$, for all branch sets $A_1, \ldots, A_t$ the
  induced subgraph $G[A_i]$ is a tree, and for all such sets with
  $|A_i| \ge 3$, the tree must have two non-adjacent leaves $v$ and
  $u$.  By \cref{lem:deg1inbranch}, both $v$ and $u$ have a private
  branch set on the other side, so we can label each branch set
  $A_i$ with $|A_i| \ge 3$ with an unordered pair $\{B_j, B_k\}$ so that $B_j$ is private
  to $v$ and $B_k$ is private to~$u$. 
  
  Now, because
  $t-1 = 2 \binom{p}{2}+1$, there must exist three branch sets $A_a$,
  $A_b$, and $A_c$ that are labeled with the same pair $\{B_j, B_k\}$.
  By contracting both $B_j$ and $B_k$ into a single vertex and taking
  the paths in $A_a$, $A_b$, and $A_c$ between the corresponding
  leaves, we obtain $K^{*}_{2,3}$ as an induced minor.  By
  \cref{lem:subk23im}, $G$ must contain a theta since the theta is the
  only triangle-free 3PC.
\end{proof}

\section{Proof of \cref{th:k343PC}}
\label{sec:33}

We start with an improvement of \cref{l:OnePath}.

\begin{lemma}
  \label{l:allPath}
  Let $G$ be a 3PC-free graph and $k\geq 2$ be an
  integer. Suppose that $X$, $Y$, $Z$ and $A_1$, \dots, $A_k$ are
  disjoint connected subsets of $V(G)$, $X$, $Y$ and $Z$ are pairwise
  anticomplete, $A_1$, \dots, $A_k$ are pairwise
  anticomplete, and for every $i\in \{1, \dots, k\}$, $A_i$ sees
  $X$, $Y$ and $Z$.

  Then the $A_i$'s are all of type path with respect to $X$, $Y$ and $Z$
  and furthermore, one of $X$, $Y$ and $Z$ is a center for all of
  them.
\end{lemma}

\begin{proof}
  By \cref{l:OnePath} applied to $A=A_i$, $B=A_{j}$ for some $j\neq i$
  and $X$, $Y$ and $Z$, we see that all $A_i$'s are of type path with
  respect to $X$, $Y$ and $Z$. It remains to prove that they all share
  a common center.

  We denote by $\tau(A_i)$ the sets of all elements
  $U \in \{X, Y, Z\}$ such that $A_i$ is centered at $U$.  Note that
  for all $i\in \{1, \dots, k\}$, $\tau(A_i)$ is nonempty. It is
  enough to prove that for all $i, j \in \{1, \dots, k\}$, either
  $\tau(A_i) \subseteq \tau(A_j)$ or $\tau(A_j) \subseteq \tau(A_i)$.
  Indeed, this implies that the sets $\tau(A_i)$ are linearly ordered
  by the inclusion, so that $\cap_{i\in \{1, \dots, k\}} \tau (A_i)$
  is non-empty and contains the common center that we are looking for.
 
  So suppose for a contradiction that $A= A_i$ and $B=A_j$ are such
  that $\tau(A_i)$ and $\tau(A_j)$ are inclusion-wise
  incomparable. So, up to symmetry, $A$ is centered at $X$ and not at
  $Y$ while $B$ is centered at $Y$ and not at $X$.

  Since $A$ is centered at $X$, there exists $P= u\dots u'$ in $A$
  such that $u$ sees $Y$, $u'$ sees $Z$, $P$ sees $X$, $P\sm u$ is
  anticomplete to $Y$ and $P\sm u'$ is anticomplete to $Z$.  Since $B$
  is centered at $Y$, there exists $Q= v\dots v'$ in $B$ such that $v$
  sees $X$, $v'$ sees $Z$, $Q$ sees $Y$, $Q\sm v$ is anticomplete to
  $X$ and $Q\sm v'$ is anticomplete to $Z$.

  Let $R = z\dots z'$ be a shortest path in $Z$ such that
  $u'z \in E(G)$ and $z'v'\in E(G)$. Let $S = y\dots y'$ be a shortest
  path in $Y$ such that $y$ sees $Q$ and $y'u\in E(G)$.  Let
  $T= x\dots x'$ be a shortest path in $X$ such that $x$ sees $P$ and
  $x'v\in E(G)$.  Observe that each of $P$, $Q$, $R$, $S$ and $T$ can
  be of length~0.

  Let $a$ be the neighbor of $x$ in $P$ closest to $u$ along $P$.  Let
  $a'$ be the neighbor of $x$ in $P$ closest to $u'$ along $P$.  Let
  $b$ be the neighbor of $y$ in $Q$ closest to $v$ along $Q$.  Let
  $b'$ be the neighbor of $y$ in $Q$ closest to $v'$ along $Q$.  See
  Figure~\ref{f:allPath}.

  \begin{figure}[!ht]
  \begin{center}
        \includegraphics{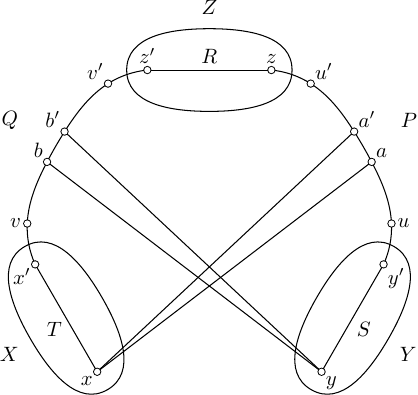}
  \end{center}
  \caption{Paths $P$, $Q$, $R$, $S$ and $T$ in the proof of Lemma~\ref{l:allPath}\label{f:allPath}}
  \end{figure}

  Suppose first that $a=a'$. Observe that $a=a' \neq u$ for otherwise
  $A$ would be centered at $Y$, contrary to our assumption. Hence,
  $ay\notin E(G)$.  If $b=b'$, then $P$, $Q$, $R$, $S$ and $T$ form a
  theta from $a$ to $b$, so $b\neq b'$.  If $bb'\in E(G)$, then $P$,
  $Q$, $R$, $S$ and $T$ form a pyramid from $a$ to $ybb'$.  If
  $bb'\notin E(G)$, then $P$, $vQb$, $b'Qv'$, $R$, $S$ and $T$ form a
  theta from $a$ to $y$ (because $ay\notin E(G)$ as already noted).
  Hence $a\neq a'$, and symmetrically we can prove that $b\neq b'$.

  Suppose now $aa'\in E(G)$.  If $bb'\in E(G)$, then $P$, $Q$, $R$,
  $S$ and $T$ form a prism from $xaa'$ to $ybb'$.  If
  $bb'\notin E(G)$, then $P$, $vQb$, $b'Qv'$, $R$, $S$ and $T$ form a
  pyramid from $y$ to $xaa'$.  Hence $aa'\notin E(G)$, and symmetrically we
  can prove that $bb'\notin E(G)$.

  We are left with the case where $a\neq a'$, $aa'\notin E(G)$,
  $b\neq b'$ and $bb'\notin E(G)$. Then $uPa$, $a'Pu'$, $vQb$,
  $b'Qv'$, $R$, $S$ and $T$ form a theta from $x$ to~$y$.
\end{proof}

The following lemma describes what happens when $K_{3, 3}$ is an
induced minor of some 3PC-free graph.  It is worth noting that it has
a true converse that we do not need to state formally. More precisely,
if in any graph six paths $A$, $B$, $C$, $P$, $Q$ and $R$ satisfy all
the properties described in \cref{l:exact1}, then they form a model
for a $K_{3, 3}$ induced minor, and moreover the graph that they
induce can be checked to be 3PC-free, see \cref{f:situations}.

Note that the statement of \cref{l:exact1} is not completely
symmetric. Namely, $A$, $B$ and $C$ are assumed to be minimal while
$X$, $Y$ and $Z$ are not.  This yields a slightly stronger statement
which is needed for the application in the proof of \cref{th:k343PC}.

\begin{lemma}
  \label{l:exact1}
  Let $G$ be a 3PC-free graph and $A$, $B$, $C$, $X$, $Y$ and $Z$ be
  connected disjoint subsets of $V(G)$ such that $X$, $Y$ and $Z$ are
  pairwise anticomplete, $A$, $B$ and $C$ are pairwise anticomplete
  and each of $A$, $B$ and $C$ sees each of $X$, $Y$ and $Z$.  Suppose
  that no connected proper subset of $A$ (resp.\ $B$ and $C$) sees
  each of $X$, $Y$ and $Z$.  Then, there exist six vertex-disjoint
  paths $A' = a\dots a'$, $B'= b\dots b'$, $C'=c\dots c'$,
  $P= p\dots p'$, $Q = q\dots q'$ and $R = r\dots r'$ in $G$ such
  that:

  \begin{itemize}
  \item Each of $A'$, $B'$ and $C'$ \emph{is equal to} exactly one of
    $A$, $B$ or $C$.
  \item Each of $X$, $Y$ and $Z$ \emph{contains} exactly one of
    $P$, $Q$ or $R$.
  \item $H= aA'a'rRr'c'C'cp'Ppa$ is a hole.
  \item $B'\sm b$ (resp.\ $B'\sm b'$, $Q\sm q$, $Q\sm q'$) is
    anticomplete to $P$ (resp.\ $R$, $A'$, $C'$).
  \item $B'$ and $Q$ both have length at most~1 (so they each contain
    at most two vertices).  
  \item Each of $P$, $R$, $A'$ and $C'$ contains at least three
    vertices and $b$ (resp.\ $b'$, $q$, $q'$) has at least two
    neighbors in $P$ (resp.\ $R$, $A'$, $C'$). 
  \item $B'$ is complete to $Q$, or $G[B'\cup Q]$ has
    four vertices and five edges.
  \end{itemize} 
\end{lemma}

\begin{proof}
  By \cref{l:allPath}, $A$, $B$ and $C$ are of type path with
  respect to $X$, $Y$ and $Z$, centered at some $Y'\in \{X, Y, Z\}$.  
  By \cref{l:allPath}, $X$, $Y$ and $Z$ are of type path with
  respect to $A$, $B$ and $C$, centered at some $B' \in \{A, B,
  C\}$. We let $X'$, $Z'$, $A'$ and $C'$ be such that $\{A', B', C'\}
  = \{A, B, C\}$ and $\{X', Y', Z'\} = \{X, Y, Z\}$.    

  So, $A'$ contains some path that sees $X'$, $Y'$ and $Z'$ as in the
  definition of type path centered at $Y'$, but by the assumption
  about the minimality of $A$, $B$ or $C$, we see that this path is in
  fact $A'$ itself. So $A'$ is equal to exactly one of $A$, $B$ or
  $C$. The arguments work also with $B$ and $C$, so that
  $A'=a\dots a'$, $B'= b\dots b'$, $C'= c \dots c'$, each of $a$, $b$
  and $c$ sees $X'$, each of $a'$, $b'$ and $c'$ sees~$Z'$, each of
  $A'$, $B'$ and $C'$ sees $Y'$, each of $A'\sm a$, $B'\sm b$ and
  $C'\sm c$ is anticomplete to $X'$ and each of $A'\sm a'$, $B'\sm b'$
  is and $C'\sm c'$ is anticomplete to $Z'$.

  Also $X'$ contains a path $P = p\dots p'$ such that $p$ sees $A'$,
  $p'$ sees $C'$, $P\sm p$ is anticomplete to $A'$, $P\sm p'$ is
  anticomplete to $C'$ and $P$ sees $B'$. Note that we cannot claim
  that $X'=P$, since we made no assumption about the minimality of
  $X$.  But since $A'\sm a$ is anticomplete to $X'$ and $P\sm p$ is
  anticomplete to $A'$, the only possible edge between $P$ and $A'$ is
  $pa$.  Similarly, $p'c$ is the only edge between $P$ and $C'$.
  Moreover, $P$ sees $B'$, and since $B\sm b$ is anticomplete to $X'$,
  we know that $b$ sees $P$ (and not only $X'$).

  Similarly, $Z'$ contains a path $R = r\dots r'$ with $ra'\in E(G)$,
  $r'c'\in E(G)$, $R\sm r$ is anticomplete to $A'$, $R\sm r'$ is
  anticomplete to $C'$ and such that $b'$ sees~$R$.  Observe that
  $A'$, $P$, $C'$ and $R$ form a hole $H$.

  Also $Y'$ is of type path with respect to $A$, $B$ and $C$ and
  centered at $B'$. So, $Y'$ contains a path $Q = q\dots q'$ such that
  $q$ sees $A'$, $q'$ sees $C'$, $Q\sm q$ is anticomplete to $A'$,
  $Q\sm q'$ is anticomplete to $C'$ and $Q$ sees $B'$.

  Let $\alpha$ be the neighbor of $q$ in $A'$ closest to $a$ along
  $A'$.  Let $\alpha'$ be the neighbor of $q$ in $A'$ closest to $a'$
  along $A$.  Let $\beta$ be the neighbor of $b$ in $P$ closest to $p$
  along~$P$.  Let $\beta'$ be the neighbor of $b$ in $P$ closest to
  $p'$ along $P$.  Let $\gamma$ be the neighbor of $q'$ in $C'$
  closest to $c$ along $C'$.  Let $\gamma'$ be the neighbor of $q'$ in
  $C'$ closest to $c'$ along $C'$.  Let $\delta$ be the neighbor of
  $b'$ in $R$ closest to $r$ along $R$.  Let $\delta'$ be the neighbor
  of $b'$ in $R$ closest to $r'$ along $R$.  See \cref{f:exact1}.

  \begin{figure}[!ht]
    \begin{center}
      \includegraphics{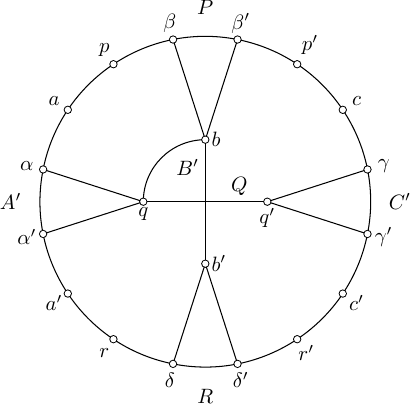}
    \end{center}
    \caption{Paths $A'$, $B'$, $C'$, $P$, $Q$ and $R$ in the proof of
      Lemma~\ref{l:exact1}\label{f:exact1}}
  \end{figure}
  
  Suppose that $B'$ has length at least~2. Then $B$ and $H$ contains a
  theta, a prism or a pyramid, namely from $\beta$ (if $\beta= \beta'$)
  or $b\beta\beta'$ (if $\beta\beta' \in E(G)$) or $b$ (otherwise), to
  $\delta$ (if $\delta = \delta'$) or $b'\delta\delta'$ (if
  $\delta\delta' \in E(G)$) or $b'$ (otherwise).  Hence $B'$ has length
  at most~1, meaning that either $b=b'$ or $bb'\in E(G)$.  Similarly,
  $Q$ has length at most~1 and $q=q'$ or $qq'\in E(G)$.

  Suppose that $\beta=\beta'$.  Then $b=b'$ for otherwise $B$ and $H$
  contain a theta (if $\delta=\delta'$), or a pyramid from $\beta$ to
  $b'\delta\delta'$ (if $\delta\delta'\in E(G)$), or a theta from
  $\beta$ to $b'$ (otherwise). Since $B'$ sees $Q$, $b$ has a neighbor
  in $Q$.  If $b$ has a unique neighbor in $Q$, say $q$ up to symmetry
  (so either $q=q'$, or $q\neq q'$ and $bq'\notin E(G)$), then the
  three paths $\beta b q$, $\beta P p a A' \alpha q$ and
  $\beta P p' c C' \gamma q' q$ form a theta from $\beta$ to $q$.  So,
  $b$ has two neighbors in $Q$. Hence, the three paths $\beta b$,
  $\beta P p a A' \alpha q$ and $\beta P p' c C' \gamma q'$ form a
  pyramid from $\beta$ to $bqq'$.  We proved that $\beta\neq \beta'$.

  Suppose that $\beta\beta'\in E(G)$.  Then $b=b'$ as otherwise $B$
  and $H$ contains a pyramid from $\delta$ to $b\beta\beta'$ (if
  $\delta=\delta'$), a prism from $b\beta\beta'$ to $b'\delta\delta'$
  (if $\delta\delta'\in E(G)$) or a pyramid from $b'$ to
  $b\beta\beta'$ (otherwise). Since $B'$ sees $Q$, $b$ has a neighbor
  in $Q$.  If $b$ has a unique neighbor in $Q$, say $q$ up to symmetry
  (so either $q=q'$, or $q\neq q'$ and $bq'\notin E(G)$), then the
  three paths $qb$, $q \alpha A' a p P \beta$ and
  $q q' \gamma C' c p' P \beta'$ form a pyramid from $q$ to
  $b\beta \beta'$.  So, $b$ has two neighbors in $Q$. Hence, the three
  paths $b$, $q \alpha A' a p P \beta$ and
  $q' \gamma C' c p' P \beta'$ form a prism from $bqq'$ to
  $b\beta\beta'$.  We proved that $\beta\beta' \notin E(G)$.  This
  implies that $P$ contains at least three vertices and $b$ has at
  least two neighbors in $P$.

  
  By a symmetric argument, we can prove that $R$ (resp.\ $A'$, $C'$)
  contains at least three vertices and $b'$ (resp.\ $q$, $q'$) has at
  least two neighbors in it.  It remains to prove that $B'$ is
  complete to $Q$, or $B'\cup Q$ induces a graph with four vertices
  and five edges.  So suppose that $B'$ is not complete to $Q$.

  
  Since $B'$ sees $Q$ and $B'$ is not complete to $Q$, there is at
  least one edge and at least one non-edge with ends in $B$ and $Q$.
  So, there must be a vertex, either in $B$ or $Q$, that is incident
  to such an edge and such a non-edge. Up to symmetry, we assume that
  this vertex is $b$ and $bq\in E(G)$ (so $bq'\notin E(G)$ and
  $q\neq q'$).  If $b=b'$, or if $b\neq b'$ and $G[B'\cup Q]$ has only
  three edges (namely $bb'$, $qq'$ and $bq$), then $bqq'$,
  $b\beta' P p' c C' \gamma q'$ and
  $b b'\delta' R r' c' C' \gamma' q'$ form a theta from $b$ to $q'$.
  We proved that $b\neq b'$ and $G[B'\cup Q]$ has at least four edges.
  So, $G[B'\cup Q]$ has four vertices and it remains to prove that it
  has exactly five edges.  So suppose for a contradiction that
  $G[B'\cup Q]$ has exactly four edges.

  If $b'q'\in E(G)$ (and therefore $b'q\notin E(G)$), then $bqq'$,
  $bb'q'$ and $b\beta' P p' c C' \gamma q'$ form a theta from $b$ to
  $q'$.  If $b'q\in E(G)$ (and therefore $b'q'\notin E(G)$), then
  $q'q$, $q' \gamma C' c p' P \beta' b$ and
  $q' \gamma' C' c' r' R \delta' b'$ form a pyramid from $q'$ to
  $bb'q$.
\end{proof}

\begin{lemma}
  \label{l:exact2}
  Let $G$ be a 3PC-free graph and $A$, $B$, $C$, $X$, $Y$ and $Z$ be
  connected disjoint subsets of $V(G)$ such that $X$, $Y$ and $Z$ are
  pairwise anticomplete, $A$, $B$ and $C$ are pairwise anticomplete,
  and each of $A$, $B$ and $C$ sees each of $X$, $Y$ and $Z$.  Suppose
  that no connected proper subset of $A$ (resp.~$B$ and $C$) sees each
  of $X$, $Y$ and $Z$.  Then exactly one of $A$, $B$ and $C$ contains
  at most~2 vertices.
\end{lemma}

\begin{proof}
  Follows directly from \cref{l:exact1}.
\end{proof}

We may now prove \cref{th:k343PC} that we restate.

\longer*

\begin{proof}
  Suppose for a contradiction that a 3PC-free graph $G$ contains
  $K_{3, 4}$ has an induced minor. So, $G$ contains seven disjoint
  connected sets $X$, $Y$, $Z$, $A$, $B$, $C$ and $D$ such that $X$,
  $Y$ and $Z$ are pairwise anticomplete, $A$, $B$, $C$ and $D$ are
  pairwise anticomplete, and each of $X$, $Y$ and $Z$ sees each of
  $A$, $B$, $C$ and $D$.  We suppose that these sets are chosen
  subject to the minimality of $A \cup B \cup C \cup D$.  It follows
  that no proper connected subset of $A$ (resp.~$B$, $C$, $D$) sees
  $X$, $Y$ and $Z$ (note that it is important here that no assumption is made about the minimality of $X$, $Y$ and $Z$).

  By \cref{l:exact2} applied to $A$, $B$, $C$, $X$, $Y$ and $Z$,
  exactly one of $A$, $B$ and $C$ has size at most~2, say
  $|A| \leq 2$, $|B|>2$ and $|C|>2$. Hence, by \cref{l:exact2}
  applied to $B$, $C$, $D$, $X$, $Y$ and $Z$, we have $|D|\leq
  2$. Hence, $A$, $B$, $D$, $X$, $Y$ and $Z$ contradict
  \cref{l:exact2}.
\end{proof}

\section{Open questions}
\label{sec:open}

We need pyramids in \cref{th:k343PC} only for the sake of the precise
description of \cref{l:exact1}, see \cref{f:3PC-pyramid}.  We
therefore wonder whether pyramids in \cref{th:k343PC} are really
needed. More precisely, we do not know whether a (theta, prism)-free
graph that contains $K_{3, 4}$ as an induced minor exists.

A \emph{wheel} is a graph made of a hole called the \emph{rim}
together with a vertex called the \emph{center} that has at least three
neighbors on the rim.  It is \emph{even} if the center has an even
number of neighbors in the rim. It is well known (and easy to check)
that even-hole-free graphs contain no prisms, no thetas and no even
wheels as induced subgraphs, since each of these configurations
implies the presence of an even hole. Conversely, many theorems about
even-hole-free graphs suggest that (theta, prism, even wheel)-free
graphs, that are called \emph{odd signable graphs}, capture the
essentials structural properties of even-hole-free graphs, see~\cite{vuskovic:evensurvey,vuskovic:truemper}. We believe
that the following is true.

\begin{conjecture}
  \label{conj:es}
  If $G$ is an odd signable graph (in particular if $G$ is an even-hole-free
  graph), then $G$ does not contain $K_{3, 3}$ as an induced minor.
\end{conjecture}

Observe that we do not know whether the even-hole-free layered wheels
contain $K_{3, 3}$ as an induced minor.
A proof of \cref{conj:es} would imply that they do not.
We also propose the following.

\begin{conjecture}
  \label{conj:k6}
  If $G$ contains $K_6$ as a minor, then $G$ contains a triangle (as a
  subgraph) or $G$ contains $K_{3, 3}$ as an induced minor. 
\end{conjecture}

In~\cite{DBLP:journals/dm/CameronSHV18}, it is proved that
triangle-free odd-signable graphs (in particular (triangle,
even-hole)-free graphs) have treewidth at most~5 and therefore do not
contain $K_6$ as a minor. So, provided that \cref{conj:es} is true,
\cref{conj:k6} is just a more precise statement.

Here are some remarks about \cref{conj:k6}. It is false with a $K_5$
assumption instead of a $K_6$ assumption, see \cref{f:K5minor} where
an (even hole, triangle)-free graph with a $K_5$ minor, first
discovered in~\cite{DBLP:journals/jgt/ConfortiCKV00}, is represented.
It is false with a ``$K_{3,4}$ as an induced minor or triangle''
conclusion because of the layered wheels.  Provided that
\cref{conj:es} is true, it is false with a ``$K_{3,3}$ as an induced
minor or 3PC as an induced subgraph'' conclusion, or with a
``$K_{3,3}$ as an induced minor or $K_4$ as subgraph'' conclusion,
because of the even-hole-free layered wheels that are pyramid-free and
$K_4$-free by \cref{th:ehf-lw}.  \Cref{conj:k6} would therefore
provide a statement that is best possible in many ways.

\begin{figure}[!ht]
  \begin{center}
    \includegraphics[scale=1.6]{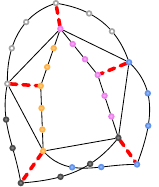}
  \end{center}
  \caption{A $K_5$ minor in an (even hole, triangle)-free graph.
    To see even hole-freeness first note that no even hole can contain a fat red edge.
  \label{f:K5minor}}
\end{figure}

It might be interesting to study the implications of a $K_{3, 3}-e$
induced minor in an even-hole free graph, where $K_{3, 3}-e$ is the
graph obtained from $K_{3, 3}$ by removing one edge.
In \cref{f:K33-eminor}, an even-hole-free graph that contains $K_{3, 3}-e$
as an induced minor is represented, and we observe that this graph plays an important role in the structural study of even-hole-free graphs,
see~\cite{vuskovic:evensurvey}.
In \cref{f:turtle}, another example of
a graph that contains $K_{3, 3}-e$ as an induced minor is represented, and we observe that  this graph contains an even wheel.

\begin{figure}[!ht]
  \begin{center}
    \includegraphics[scale=1.3]{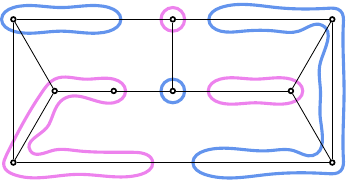}
  \end{center}
  \caption{A $K_{3, 3}-e$ induced minor in an (even hole,
    triangle)-free graph\label{f:K33-eminor}}
\end{figure}

\begin{figure}[!ht]
  \begin{center}
    \includegraphics[scale=1.4]{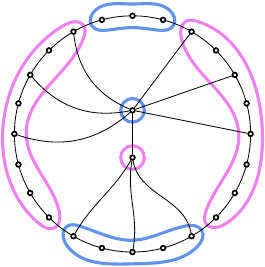}
  \end{center}
  \caption{A $K_{3, 3}-e$ induced minor in a graph without triangles and without thetas\label{f:turtle}}
\end{figure}

More generally, we believe that studying implications between
different containment relations for different kinds of graphs might
have more applications.  For instance, this approach is used
in~\cite{DBLP:conf/iwoca/DallardDHMPT24} to design a polynomial
time algorithm that decides whether an input graph contains $K_{2, 3}$
as an induced minor.  We wonder what is the complexity of detecting
$K_{3, 3}$ as an induced minor.

\section{Acknowledgement}

We are grateful to Marthe Bonamy, Sepehr Hajebi, Martin Milani\v c,  Sophie Spirkl and Kristina Vu\v{s}kovi\'c for helpful discussions.

Maria Chudnovsky is supported by NSF-EPSRC Grant DMS-2120644 and by
AFOSR grant FA9550-22-1-0083.

Meike Hatzel is supported by the Federal Ministry of Education and Research (BMBF) and by a fellowship within the IFI programme of the German Academic Exchange Service (DAAD) and by the Institute for Basic Science (IBS-R029-C1).

Tuukka Korhonen is supported by the Research Council of Norway via the project BWCA (grant no. 314528).

Nicolas Trotignon is partially supported by the French National Research Agency under research grant ANR DIGRAPHS ANR-19-CE48-0013-01 and the LABEX MILYON  (ANR-10-LABX-0070) of Université de Lyon, within the program Investissements d’Avenir (ANR-11-IDEX-0007) operated by the French National Research Agency (ANR).

Part of this  research was conducted during two Graph Theory Workshops  at the McGill 
University Bellairs Research Institute, and we express our gratitude to the institute and to the organizers of the workshop.
 
 Part of this work was done when Nicolas Trotignon visited Maria Chudnovsky at Princeton University with generous support of the H2020-MSCA-RISE project CoSP- GA No. 823748.


\end{document}